\documentclass[11pt]{article}
\usepackage{amsfonts}
\usepackage{latexsym}
\usepackage{amsmath}
\usepackage{graphicx}
\usepackage{amssymb}
\usepackage{amsthm}
\usepackage{amsmath}
\usepackage{fancyhdr}
\usepackage{mathrsfs}
\usepackage{color}
\numberwithin{equation}{section} \allowdisplaybreaks

\newtheorem{theorem}{Theorem}[section]
\newtheorem{lemma}[theorem]{Lemma}
\newtheorem{remark}[theorem]{Remark}

\newtheorem{proposition}[theorem]{Proposition}

\newtheorem{assumption}[theorem]{Assumption}

\definecolor{darkgreen}{rgb}{0,.75,0}
\newcommand{\C}{\mathcal{C}}
\newcommand{\Lc}{\mathcal{L}}
\newcommand{\N}{\mathbb{N}}
\newcommand{\norm}[1]{\big\|#1\big\|}
\newcommand{\E}{\mathbb{E}}

\begin{document}
\title{\textbf{Bayesian Posterior Contraction Rates for Linear Severely Ill-posed Inverse Problems}
{
\thanks{Agapiou and Stuart are supported by ERC and Yuan-Xiang Zhang is supported by China Scholarship Council and the
NNSF of China (No. 11171136).}}}
\author{Sergios Agapiou $^\diamond$\quad {Andrew M. Stuart $^\diamond$\quad Yuan-Xiang Zhang $^{\star}$
} \vspace{0.5cm}\\$\diamond$ Mathematics Institute,\\ University of
Warwick,\\ Coventry CV4 7AL, United Kingdom\\\\
$\star$ School of Mathematics and Statistics,\\ Lanzhou University,\\ Lanzhou 730000, China\\
\\{\small Email: S.Agapiou@warwick.ac.uk,\quad}{\small  A.M.Stuart@warwick.ac.uk, }\\
{\small ZhangYuanxiang@lzu.edu.cn}}
\date{}
\maketitle

{\bf Abstract} {We consider a class of linear ill-posed inverse problems
arising from inversion of a compact operator with singular
values which decay exponentially to zero. We adopt a
Bayesian approach, assuming a Gaussian prior on the unknown
function. The observational noise is assumed to be
Gaussian; as a consequence the prior is conjugate to the likelihood so that
the posterior distribution is also Gaussian. We study
Bayesian posterior consistency in the small observational noise
limit. We assume that the forward operator and the
prior and noise covariance operators commute with one another.
We show how, for given smoothness
assumptions on the truth, the scale parameter of the
prior, which is a constant multiplier of the prior covariance operator, can be adjusted to optimize the rate of posterior
contraction to the truth, and we explicitly compute
the logarithmic rate.}
\vspace{0.4cm}



\vspace{0.4cm}

\section{Introduction}
\label{s:1}

Let $H$ be an infinite dimensional separable Hilbert space
{and let $\mathcal {L}: H \rightarrow H$ be an injective
compact linear operator with non-closed range. We consider the
ill-posed inverse problem of finding $u$ from data $d$, where
\begin{equation}\label{1}
d=\mathcal {L}u+\eta,
\end{equation}
and where $\eta$ represents noise.}
The problem (\ref{1}) is called mildly
or modestly ill-posed if the singular values of the forward mapping
$\mathcal{L}$ decay algebraically, while it is called severely
ill-posed if the singular values of $\mathcal {L}$ decay
exponentially \cite{EHN96}. {Our interest is focussed
on the severely ill-posed case, and on the small
observational noise limit.}

The use of classical (deterministic) regularization methods
for (\ref{1}), and the small-noise limit in particular,
is well-studied in  both the mildly ill-posed \cite{EHN96}
and severely ill-posed \cite{Hoh00} cases; nonlinear
inverse problems have also been studied from
this perspective \cite{EHN96}.
However, if we wish to incorporate information
concerning the statistical structure of the unknown and the noise,
then it is natural to adopt a Bayesian perspective. The
Bayesian approach to linear ill-posed inverse problems
was adopted in \cite{Fr70},
in which the severely ill-posed
problem of inverting the heat operator was considered, and
then developed systematically in \cite{Man84,LPS89}. More
recently, nonlinear inverse problems have been given a
Bayesian  formulation \cite{La07,Stu10,las12,las12b}.
However, study of the small noise limit, known as posterior
consistency in the Bayesian context, is an under-developed
aspect of the Bayesian methodology for inverse problems.
Our work adds to the growing literature in this area.

For mildly ill-posed linear problems, subject to Gaussian
observational noise, Bayesian posterior consistency is
considered in the recent papers \cite{ALS12, KVZ11}. In
\cite{KVZ11}, sharp contraction rates are obtained for
white observational noise when the forward operator
$\mathcal {L}$ and the prior covariance operator are
simultaneously diagonalizable; this allows the analysis to
proceed through the study of an infinite set of uncoupled
scalar linear inverse problems. In \cite{ALS12} the
setting of \cite{KVZ11} is generalized to allow for
non-white noise and operators which are not simultaneously
diagonalizable, using tools from PDE theory.
The paper \cite{KVZ12} is, to the best of the authors'
knowledge, the first to study Bayesian posterior consistency
for severely ill-posed problems. It concerns the
one-dimensional backward heat equation with white noise,
where the $j$th eigenvalue of the (self-adjoint) forward mapping
decays like $\exp({-j^2})$ and works in the simultaneously
diagonalizable paradigm of \cite{KVZ11}. In this paper,
we generalize the work in \cite{KVZ12} by studying Bayesian
posterior consistency for a class of severely ill-posed inverse
problems in which the $j$th singular value of
$\mathcal{L}$ decays as $\exp({-sj^b})$ for arbitrary positive
$s$ and $b$, again working in the simultaneously
diagonalizable paradigm of \cite{KVZ11}. In addition to the backward heat equation
considered in \cite{KVZ12} ($b=2$), there are a variety of
ill-posed inverse problems covered by our theory.
For instance, the Cauchy problem for the Laplace equation
and  the Cauchy problem for the Helmholtz equation or the
modified Helmholtz equation (see \cite{ZFD12} and the references therein):
the eigenvalue decay of the forward mapping for these
three examples  corresponds to $b=1$. Our analysis
is inspired by both the problem and techniques used
in \cite{KVZ12}; however our generalized setting leads
to some technical improvements in the proofs, we discuss
new results relating to the equivalence of the prior
and posterior and we include a numerical illustration
for the Helmholtz equation.

The rest of this paper is organized as follows. In
Section \ref{s:2} we introduce notation and give informal calculations for the
posterior mean and covariance operator. In Section \ref{s:3} we
characterize the posterior distribution rigorously
and show that it is equivalent, in the sense of measures,
to the prior -- see Theorems \ref{t3.1}
and \ref{t3.2}.
In Section \ref{s:4} we present and prove the main results
concerning posterior consistency,
characterizing the error in the mean in
Theorem {\ref{t4.1}, the contraction of the posterior
covariance in Theorem {\ref{t4.2}} and putting these together
to estimate posterior contraction rates in Theorem
{\ref{t4.3}}.} A discussion of the convergence rates obtained in our three main theorems, which includes comments on their minimax optimality, is contained in Remark \ref{ch2:rem}. Some technical lemmas which are essential to
the proof of Theorems {\ref{t4.1}}, {\ref{t4.2}} and
{\ref{t4.3}} are attached at the end of this section. Section \ref{s:5}
concludes the paper with a simple example for which the
theoretical analysis can be applied and includes
a numerical experiment which is consistent with the theory.

\section{Notation and Problem Setting}
\label{s:2}

\subsection{Notation}

Throughout the paper, $\langle\cdot,\cdot\rangle$ and $\|\cdot\|$
denote the inner product and norm of the separable infinite dimensional Hilbert space $H$. For a
self-adjoint positive operator $\Gamma$, we define the weighted
inner product and the corresponding norm as follows,
\begin{equation*}
\langle\cdot,\cdot\rangle_{\Gamma}=\langle\Gamma^{-\frac{1}{2}}\cdot,
\Gamma^{-\frac{1}{2}}\cdot\rangle,\quad
\|\cdot\|_{\Gamma}=\|\Gamma^{-\frac{1}{2}}\cdot\|.
\end{equation*}

Let $\{\varphi_j\}_{j=1}^\infty$ denote an orthonormal basis in $H$.
Then we can express $u\in H$ as $u=\sum\limits_{j=1}^\infty
u_j\varphi_j$ with $u_j=\langle u,\varphi_j\rangle$ and for
$\gamma\geq0$ we define the norm $\|.\|_\gamma$ by
\[\|u\|^2_\gamma:=\sum\limits_{j=1}^{\infty}u_j^2j^{2\gamma}.\] We use
$H^\gamma, \;\gamma\geq0$ to denote the Sobolev-like space
\begin{equation*}
H^\gamma=\{u\in H: \|u\|_\gamma<\infty\}.
\end{equation*} For $\gamma<0$, we define the spaces $H^\gamma$ by duality: $H^\gamma=(H^{-\gamma})^\ast$.

In the following we consider random variables drawn from Gaussian disrtibutions in $H$,
denoted by $N(\theta,\Sigma)$ where the mean $\theta$ is an element of $H$ and the covariance
operator $\Sigma$ is a positive definite, self-adjoint, trace class, linear operator in $H$.
The operator $\Sigma$ possesses an infinite set of eigenfunctions $\{\varphi_j\}_{j\in\N}$
which correspond to positive eigenvalues $\{\sigma_j\}_{j\in\N}$ and which form an orthonormal
basis of $H$. One can express a draw $y$ from $N(\theta,\Sigma)$ using the Karhunen-Loeve expansion as
\begin{equation}y=\theta+\sum_j\sqrt{\sigma_j}\xi_j\varphi_j,\end{equation}
where $\xi_j$ are independent and identically distributed $N(0,1)$ real random variables, \cite{Giu06, Stu10}.
In particular, the expansion coefficients $y_j=\theta_j+\sqrt{\sigma_j}\xi_j$ are $N(\theta_j,\sigma_j)$
real variables and it is easy to see that $\E\norm{y}^2=\norm{\theta}^2+\rm Tr(\Sigma)$ and that
for any bounded linear operator $T$ in $H$, $Ty$ is distributed as $N(T\theta, T\Sigma T^\ast)$.  {It is also straightforward to check that if $\theta=0$ and $\sigma_j=j^{-2r}$ for some $r\in\mathbb{R},$ then $y\in H^{\gamma}$ almost surely, for any $\gamma<r-\frac12$.}

{For two sequences $k_j$ and $h_j$ of real numbers,
$k_j\asymp h_j$ means
that $\frac{|k_j|}{|h_j|}$ is bounded away from zero and infinity as
$j\rightarrow\infty$, $k_j\lesssim h_j$ means that $\frac{k_j}{h_j}$
is bounded as $j\rightarrow\infty$,
and $k_j\sim h_j$ means that $\frac{k_j}{h_j}\rightarrow
1$ as $j\rightarrow\infty$.}
We will use $M$ to denote a constant which is different from
occurrence to occurrence.

\subsection{Bayesian setting and informal charaterization
of the posterior}

In this subsection we describe the assumptions underlying the
Bayesian formulation of the linear inverse problem.
Furthermore we provide informal calculations which
motivate the expressions for the posterior mean and covariance.
These will be made precise in Section \ref{s:3}.

We place a scaled Gaussian prior on the unknown $u$ of the form
$\mu_0:=N(0,\tau^2\mathcal{C}_0)$, where $\tau>0$ is a scale parameter and
$\mathcal{C}_0$ is a self-adjoint, positive-definite, trace class,
linear operator on $H$. We assume Gaussian observational noise in (\ref{1})
which is independent of $u$.
In particular, we model the data as
\begin{equation}\label{21}
d=\mathcal{L}u+\frac{1}{\sqrt{n}}\xi,
\end{equation}
where $\frac{1}{\sqrt{n}}$ is a scale parameter modelling the noise level
and $\xi$ is a random variable independent of $u$ and distributed as $N(0, \C_1)$. The linear operator
$\mathcal{C}_1$ is assumed to be self-adjoint,
positive-definite, bounded, but not necessarily trace class on $H$. This allows for
the possibility of having irregular
noise which is not in $H$. For example, the case where $\xi$
is white noise corresponds to $\C_1=I$, and
can be viewed as a Gaussian random variable in $H^{-r}$
for $r>\frac12$. Under these assumptions, the conditional
distribution of $d|u$, called the \emph{data likelihood},
is the translate of $N(0,\mathcal{C}_1)$ by $\mathcal{L}u$,
which is also Gaussian:
\begin{equation}\label{22}
N(\mathcal{L}u, \frac{1}{n}\mathcal{C}_1).
\end{equation}

In finite dimensions the density of the \emph{posterior} distribution,
that is the conditional distribution of $u|d$, is found
from Bayes rule to be proportional to $\exp(-\Phi(u))$, where
\begin{equation}\label{24}
\Phi(u)=\frac{n}{2}\|d-\mathcal{L}u\|_{\mathcal{C}_1}^2+\frac{1}{2\tau^2}\|u\|_{\mathcal{C}_0}^2.
\end{equation}
This suggests that in our infinite dimensional setting, the posterior distribution is
Gaussian, $\mu^d:=N(m,\C)$, where the mean $m$ and covariance $\C$ can be informally
derived from (\ref{24}) using completion of the square:
\begin{equation}\label{26}
\mathcal{C}^{-1}=n\mathcal{L}^*\mathcal{C}_1^{-1}\mathcal{L}+\frac{1}{\tau^2}\mathcal{C}_0^{-1},
\end{equation}
and
\begin{equation}\label{27}
\frac{1}{n}\mathcal{C}^{-1}m=\mathcal{L}^*\mathcal{C}_1^{-1}d.
\end{equation}
{Note that in general the last two formulae need to be interpreted weakly using the Lax-Milgram theory as in \cite{ALS12}. However, in the present paper we work in a diagonal setup which makes the handling of the unbounded inverse covariance operators straightforward.}

Observe that the posterior mean $m$ is the minimizer of the
functional $\Phi(u)$. If we define
$\Phi_0(u)=\frac{1}{n}\Phi(u)$ and denote
\begin{equation}\label{28}
\lambda:=\frac{1}{n\tau^2},
\end{equation}
then $m$ also minimizes the
functional $\Phi_0(u)$, that is,
\begin{equation}\label{29}
m=\arg\min_{u}\Phi_0(u),
\end{equation}
where \[\Phi_0(u)=\frac{1}{2}\|d-\mathcal{L}u\|_{\mathcal{C}_1}^2+
\frac{\lambda}{2}\|u\|_{\mathcal{C}_0}^2.\] {Thus the posterior mean
is a Tikhonov-Phillips regularized solution in the classical sense {(in fact $\Phi_0$ is almost surely infinite and we should really consider $\Psi_0=\Phi_0-\frac12\norm{d}^2_{\C_1}$ which is finite; the minimizer is unaffected)}.}
This reveals {the close} connection between Bayesian and classical
regularization for inverse problems. In the deterministic framework,
$\lambda$ is called the \emph{regularization parameter} which is
carefully chosen in order to balance consistency and stability.
Similarly, for given inverse noise level $n$, the scale parameter
$\tau$ introduced in the prior can be judiciously chosen to
guarantee a small error between the posterior mean and the true
unknown, as we will see in Section \ref{s:4}.

Posterior consistency refers, in statistical inverse problems, to
studying the relationship between the result of the statistical
analysis and the truth which underlies the data in either the small
noise or large data limits; we concentrate on the small noise limit.
We consider the standard Bayesian variant on frequentist posterior
consistency \cite{DF86, GGV00} for our severely ill-posed inverse
problem. To this end we consider observations which are
perturbations of the image of a fixed element $u^\dagger\in H$ by a
scaled Gaussian additive noise, that is, we have data $d=d^\dagger$
of the form
\begin{equation}\label{212}
d^\dagger=\mathcal{L}u^\dagger+\frac{1}{\sqrt{n}}\xi
\end{equation}
where $\xi$ is a single realization of $N(0,\mathcal{C}_1).$
This choice of data model gives the posterior distribution as
$\mu^{d^\dagger}_{\lambda,n}:=N(m^{\dagger},\mathcal{C})$, where
$\mathcal{C}$ is given by (\ref{26}) and $m^{\dagger}$
is given by (\ref{27}) with $d=d^\dagger.$ Similar to the
practice in the deterministic framework, we assume
a-priori known regularity of the true solution and identify
contraction rates of the posterior $\mu_{\lambda, n}^{d^\dagger}$ to
a Dirac measure centered on the true solution, as the noise disappears ($n\to\infty$).

\subsection{Model assumptions}
In this subsection we present our assumptions on the operators appearing in our
framework, that is, on the forward operator $\Lc$, the prior covariance operator $\C_0$ and the noise
covariance operator $\C_1$.

\begin{assumption}
\label{a2.1}

The operators $\mathcal{L}$, $\mathcal{C}_0$ and $\mathcal{C}_1$
commute with one another, so that $\mathcal{L}^*\mathcal{L}$,
$\mathcal{C}_0$ and $\mathcal{C}_1$ have the same eigenfunctions
$\{\varphi_j\}_{j=1}^\infty$. {The corresponding eigenvalues
$\{l_j^2\}_{j=1}^\infty$, $\{c_{0j}\}_{j=1}^\infty$ and
$\{c_{1j}\}_{j=1}^\infty$ of $\mathcal{L}^*\mathcal{L}$,
$\mathcal{C}_0$ and $\mathcal{C}_1$ are
assumed to satisfy}
\begin{equation}\label{210}
l_j\asymp\exp(-sj^b),\quad c_{0j}=j^{-2\alpha},\quad
c_{1j}=j^{-2\beta},
\end{equation}
for $s>0,b>0,\alpha>\frac{1}{2},\beta\geq 0$. Furthermore, the fixed
true solution $u^\dagger$ belongs to $H^\gamma$ for some $\gamma>0$.
\end{assumption}

\begin{remark}
\label{r2.2}

As is well known in finite dimensions, in the current infinite dimensional
separable Hilbert-space setting, if $\mathcal{L}$, $\mathcal{C}_0$ and $\mathcal{C}_1$
commute with one another, then $\mathcal{L}^*\mathcal{L}$,
$\mathcal{C}_0$ and $\mathcal{C}_1$ have the same eigenfunctions
$\{\varphi_j\}_{j=1}^\infty$ \cite{Lax07, Sch07}. 

\end{remark}

\begin{remark}
One can relax the assumptions on the eigenvalues of $\C_0$ and
$\C_1$ to $c_{0j}\asymp j^{-2\alpha}$ and $c_{1j}\asymp j^{-2\beta}$
without affecting any of the subsequent results.
\end{remark}

\section{Characterization of the Posterior}\label{s:3}
In \cite{Man84, LPS89} it is proved in the infinite dimensional setting that the posterior is Gaussian with
 covariance and mean given by
\begin{equation}\label{31}
\mathcal{C}=\tau^2\mathcal{C}_0-\tau^2\mathcal{C}_0\mathcal{L}^*
(\mathcal{L}\mathcal{C}_0\mathcal{L}^*+\lambda\mathcal{C}_1)^{-1}\mathcal{L}\mathcal{C}_0
\end{equation}
and
\begin{equation}\label{32}
m=\mathcal{C}_0\mathcal{L}^*(\mathcal{L}\mathcal{C}_0\mathcal{L}^*+\lambda\mathcal{C}_1)^{-1}d,
\end{equation}
respectively. {In general, the operator $(\Lc\C_0\Lc^\ast+\lambda\C_1)^{-1}$ in the last two formulae needs measure theoretic clarification. However, in the simultaneously diagonalizable case considered here, the interpretation is trivial and furthermore
these formulae are equivalent to the formulae (\ref{26}) and (\ref{27})
\cite[Example 6.23]{Stu10}}. Furthermore, since $\mathcal{L}$, $\mathcal{C}_0$ and
$\mathcal{C}_1$ commute with one another, the equations (\ref{31})
and (\ref{32}) can be rewritten as
\begin{equation}\label{33}
\mathcal{C}=\tau^2\mathcal{C}_0-\tau^2\mathcal{A}\mathcal{L}\mathcal{C}_0
\end{equation}
and
\begin{equation}\label{34}
m=\mathcal{A}d,
\end{equation}
where $\mathcal{A}: H\rightarrow H$ is the continuous linear
operator
\[\mathcal{A}=\mathcal{C}_0^{\frac{1}{2}}\big(\mathcal{C}_0^{\frac{1}{2}}\mathcal{L}^*
\mathcal{L}\mathcal{C}_0^{\frac{1}{2}}+\lambda\mathcal{C}_1\big)^{-1}\mathcal{C}_0^{\frac{1}{2}}\mathcal{L}^*
=\mathcal{C}_0\mathcal{L}^*(\mathcal{L}\mathcal{C}_0\mathcal{L}^*+\lambda\mathcal{C}_1)^{-1}.\] {In fact even if $d\notin H$, $\mathcal{A} d$ can be defined using the diagonalization.}

In the next two theorems we
show that the Gaussian posterior distribution
$\mu^d$, with covariance and mean given by
(\ref{33}) and (\ref{34}), is a proper conditional Gaussian
distribution on $H$ and is absolutely continuous with
respect to the prior.

\begin{theorem}
\label{t3.1}

Suppose Assumption \ref{a2.1} holds, then: (i) the covariance
operator $\mathcal{C}$ of the conditional distribution $\mu^d$ given
by (\ref{33}) is trace class on $H$; (ii) the mean $m$ of the
conditional posterior distribution given by (\ref{34}) is an element
of $H$, {almost surely with respect to the joint distribution
of $(u,d).$}
Thus $\mu^d(H)=1$.
\end{theorem}

\begin{proof}

The fact that $\mu^d(H)=1$ follows from (i) and (ii)
is well-known \cite{Giu06}. We thus prove these two points.

(i) {Using the basis $\{\varphi_j\}$, by equation (\ref{33})
we have that the eigenvalues of $\mathcal{C}$ are given by}
\begin{equation}\label{35}
c_j=\tau^2c_{0j}-\frac{\tau^2c^2_{0j}l_j^2}{c_{0j}l_j^2+\lambda
c_{1j}}=\frac{\tau^2\lambda c_{0j}c_{1j}}{c_{0j}l_j^2+\lambda
c_{1j}}\leq\tau^2c_{0j}.
\end{equation}
Since $\mathcal{C}_0$ is trace class on $H$, it follows that
$\mathcal{C}$ is trace class on $H$.

(ii) From (\ref{34}) we have that,
\begin{eqnarray}\label{36}
\mathbb{E}\|m\|^2&=&\mathbb{E}\|\mathcal{A}d\|^2
=\mathbb{E}\|\mathcal{A}\mathcal{L}u+\frac{1}{\sqrt{n}}\mathcal{A}\xi\|^2\nonumber\\
&=&\mathbb{E}\|\mathcal{A}\mathcal{L}u\|^2+\frac{1}{n}\mathbb{E}\|\mathcal{A}\xi\|^2
\end{eqnarray}
{since $\xi$ and $u$ are independent and $\xi$ has mean zero.
{In this simultaneously diagonalizable setting it is straightforward to see using the Karhunen-Loeve expansion, that even if $\xi$ is not in $H$ the distribution of
$\mathcal{A}\xi$ is $N(0,\mathcal{A}\mathcal{C}_1\mathcal{A}^*),$ which, due to the smoothness of $\mathcal{A}$, is a random variable in $H$.} It follows, again working in the
basis $\{\varphi_j\}_{j=1}^\infty$,
that
}
\begin{eqnarray}\label{37}
\mathbb{E}\|m\|^2&=&\mathbb{E}\|\mathcal{A}\mathcal{L}u\|^2+
\frac{1}{n}{{\rm Tr}}(\mathcal{A}\mathcal{C}_1\mathcal{A}^*)\nonumber\\
&=&\sum_j\frac{{\tau^2}c_{0j}^3l_j^4}{(l_j^2c_{0j}+\lambda
c_{1j})^2}+\frac{1}{n}\sum_j\frac{c_{0j}^2l_j^2c_{1j}}{(l_j^2c_{0j}+\lambda
c_{1j})^2}\nonumber\\
&\leq&\frac{{\tau^2}}{\lambda^2}\sum_jc_{0j}^3c_{1j}^{-2}l_j^4+
\frac{1}{n\lambda^2}\sum_jc_{0j}^2c_{1j}^{-1}l_j^2\nonumber\\
&\asymp&\frac{{\tau^2}}{\lambda^2}\sum_jj^{4\beta-6\alpha}\exp(-4sj^b)+
\frac{1}{n\lambda^2}\sum_jj^{2\beta-4\alpha}\exp(-2sj^b)\nonumber\\
&<&\infty.\nonumber
\end{eqnarray}
Hence $\|m\|$
is almost surely finite, which completes the proof.
\end{proof}

\begin{theorem}
\label{t3.2}

Suppose Assumption \ref{a2.1} holds, then the posterior measure $\mu^d=N(m,\C)$ with covariance
and mean given by (\ref{33}) and (\ref{34}), respectively, is
equivalent to the prior measure $\mu_0=N(0,\tau^2\mathcal{C}_0)$, {almost surely with respect to the joint distribution of $(u,d)$.}
\end{theorem}

\begin{proof}

By the Feldman-Hajek theorem \cite[Theorem 2.23]{DZ92}, to show that the Gaussian measure
$\mu^d=N(m,\mathcal{C})$ is equivalent to $\mu_0=N(0,\tau^2\C_0)$, it suffices to
show:

(i) The Cameron-Martin spaces associated with $\mu^d$ and $\mu_0$
are equal, that is,
$\mathscr{D}(\mathcal{C}^{-\frac{1}{2}})=\mathscr{D}(\mathcal{C}_0^{-\frac{1}{2}}):=E$.

(ii) The posterior mean $m$ lies in the Cameron-Martin space $E$.

(iii) The operator
$T:=I-\tau^2\mathcal{C}^{-\frac{1}{2}}\mathcal{C}_0\mathcal{C}^{-\frac{1}{2}}$
is Hilbert-Schmidt.\\

We now check the validity of the above conditions.
For (i) it is equivalent to show
that there exists a constant $M$ such that
\begin{equation}\label{38}
\langle h, \mathcal{C}h\rangle\leq M\langle h,
\mathcal{C}_0h\rangle, \forall h\in H
\end{equation}
and
\begin{equation}\label{39}
\langle h, \mathcal{C}_0h\rangle\leq M\langle h,
\mathcal{C}h\rangle, \forall h\in H;
\end{equation}
this follows from \cite[Lemma 6.15]{Stu10} using \cite[Proposition B1]{DZ92}.
Using the eigenbasis expansion, these are equivalent to
\begin{equation}\label{310}
c_{j}\leq Mc_{0j}
\end{equation}
and
\begin{equation}\label{311}
c_{0j}\leq M c_j.
\end{equation}
From (\ref{35}), we
know that (\ref{310}) is true with $M=\tau^2$. Again by (\ref{35}),
we have
\begin{eqnarray}\label{312}
c_j=\frac{\tau^2c_{0j}}{1+\lambda^{-1}l_j^2c_{0j}c_{1j}^{-1}}
\asymp\frac{\tau^2c_{0j}}{1+\lambda^{-1}\exp(-2sj^b)j^{2\beta-2\alpha}}\geq
Mc_{0j},
\end{eqnarray}
where $M=\frac{\tau^2}{1+K}$ and $K$ is a constant.

For (ii), it is easy to check that
$E=\mathscr{D}(\mathcal{C}_0^{-\frac{1}{2}})=H^\alpha$. The
{mean square expectation of the posterior mean $m$ in
$H^\alpha$} can be estimated similarly to (\ref{37}):
\begin{eqnarray}\label{313}
\mathbb{E}\|m\|_{H^\alpha}^2&&=\mathbb{E}\|\mathcal{C}_0^{-\frac{1}{2}}m\|^2=
\mathbb{E}\|\mathcal{C}_0^{-\frac{1}{2}}\mathcal{A}d\|^2\nonumber\\
&&=\mathbb{E}\|\mathcal{C}_0^{-\frac{1}{2}}\mathcal{A}\mathcal{L}u+
\frac{1}{\sqrt{n}}\mathcal{C}_0^{-\frac{1}{2}}\mathcal{A}\xi\|^2\nonumber\\
&&=\mathbb{E}\|\mathcal{C}_0^{-\frac{1}{2}}\mathcal{A}\mathcal{L}u\|^2+
\frac{1}{n}{\rm Tr}(\mathcal{C}_0^{-\frac{1}{2}}\mathcal{A}
\mathcal{C}_1\mathcal{A}^*\mathcal{C}_0^{-\frac{1}{2}})
\nonumber\\
&&=\sum_j\frac{\tau^2c_{0j}^2l_j^{{4}}}{(l_j^2c_{0j}+\lambda
c_{1j})^2}+{\lambda}\sum_j\frac{c_{0j}l_j^2c_{1j}}{(l_j^2c_{0j}+\lambda
c_{1j})^2}\nonumber\\
&&\leq\frac{\tau^2}{\lambda^2}\sum_jc_{0j}^2c_{1j}^{-2}l_j^{{4}}+
\frac{1}{{\lambda}}\sum_jc_{0j}c_{1j}^{-1}l_j^2\nonumber\\
&&\asymp\frac{\tau^2}{\lambda^2}\sum_jj^{4\beta-4\alpha}\exp(-{4}sj^b)+
\frac{1}{{\lambda}}\sum_jj^{2\beta-2\alpha}\exp(-2sj^b)\nonumber\\
&&<\infty,
\end{eqnarray}
therefore $m\in E$ almost surely.

For (iii), using (\ref{35}) again, we have
\begin{eqnarray}\label{314}
\sum_{j=1}^{\infty}{\Bigl(}1-\frac{\tau^2c_{0j}}{c_j}{\Bigr)}^2=
{\frac{1}{\lambda^2}}\sum_{j=1}^{\infty}c_{0j}^2l_j^4c_{1j}^{-2}
\asymp\sum_{j=1}^{\infty}\exp(-4sj^b)j^{4\beta-4\alpha}<\infty,
\end{eqnarray}
{demonstrating} that the operator $T$ is Hilbert-Schmidt.
\end{proof}

The preceding result is interesting because, without the
assumption that the inverse problem is severely ill-posed,
it is possible to construct
linear inverse problems of the form considered in this
paper, but for which the posterior is not absolutely continuous
with respect to the prior. For example, suppose that
we modify Assumption \ref{a2.1} so that the
forward operator $\Lc$ has singular values that decay algebraically,
$l_j\asymp j^{-\ell}$, but retain the same assumptions
on the prior and noise covariances. Then the posterior is again Gaussian
with covariance and mean given by the formulae (\ref{31}) and
(\ref{32}). The following proposition shows
that, if the noise is too smooth, then the posterior is not absolutely
continuous with respect to the prior:
\begin{proposition}
If $\beta\geq \alpha+\ell-\frac14$  then the posterior $\mu^d=N(m,\C)$ is
not absolutely continuous with respect to the prior $N(0,\tau^2\C_0)$, independently of the data $d$.
\end{proposition}
\begin{proof}
It suffices to show that the third condition of the Feldman-Hajek theorem
fails \cite[Theorem 2.23]{DZ92}. Indeed, $\C$ is diagonalizable in the
basis $\{\varphi_j\}_{j\in\N}$ with eigenvalues $c_j$ such that
\[c_j\asymp \frac{j^{-2\alpha-2\beta}}{j^{-2\beta}+j^{-2\ell-2\alpha}}.\]
Thus, the operator $T:=I-\tau^2\C^{-\frac12}\C_0\C^{-\frac12}$ is also
diagonalizable in $\{\varphi_j\}_{j\in\N}$ with eigenvalues $t_j$,
where  \[t_j=1-\frac{\tau^2c_{0j}}{c_j}\asymp j^{-2\alpha-2\ell+2\beta}.\]
Hence, the operator $T$ is Hilbert-Schmidt, if and only if the sequence $\{t_j\}$ is
square summable, that is, if and only if $\beta<\alpha+\ell-\frac14$.
\end{proof}

\section{Posterior Contraction}\label{s:4}

In this section, we study the {limiting} behavior of the
posterior distribution {$\mu^{d^\dagger}_{\lambda,n}$}
as the noise disappears,
$n \to \infty.$ Intuitively, we expect the
mass of the posterior to concentrate in a small ball centered on the
fixed true solution. {As in \cite{ALS12, KVZ11,
KVZ12, PSZ12}, we study this
problem by identifying positive numbers $\epsilon_n$ such that, for
arbitrary positive numbers $M_n\rightarrow \infty$,} there holds
\begin{equation}\label{41}
\mathbb{E}^{d^\dagger}\mu^{d^\dagger}_{\lambda,n}\{u:
\|u-u^\dagger\|\geq M_n\epsilon_n\}\rightarrow 0.
\end{equation}
Here expectation is with respect to the random variable
$d^\dagger$, with probability distribution given by
the data likelihood $N(\mathcal{L}u^\dagger, \frac{1}{n}\mathcal{C}_1)$, and
$\epsilon_n$ is called the contraction rate of the posterior
distribution with respect to the $H$-norm.

By the Chebyshev inequality, we have
\begin{equation}\label{42}
\mathbb{E}^{d^\dagger}\mu^{d^\dagger}_{\lambda,n}\{u:
\|u-u^\dagger\|\geq
M_n\epsilon_n\}\leq\frac{1}{M_n^2\epsilon_n^2}\mathbb{E}^{d^\dagger}
{\Bigl(}\int\|u-u^\dagger\|^2\mu_{\lambda,n}^{d^\dagger}(du){\Bigr)},
\end{equation}
thus if
\begin{equation}\label{43}
\mathbb{E}^{d^\dagger}
{\Bigl(}\int\|u-u^\dagger\|^2\mu_{\lambda,n}^{d^\dagger}(du){\Bigr)}\leq
M_0\epsilon_n^2,
\end{equation}
where $M_0$ is a constant, we get
that (\ref{41}) holds as $M_n\rightarrow
\infty$. The left hand side of (\ref{43}) is the squared posterior
contraction (SPC) which satisfies
\begin{equation}\label{44}
SPC=\mathbb{E}^{d^\dagger}\|m^\dagger-u^\dagger\|^2+{\rm Tr}(\mathcal{C}),
\end{equation}
{and} therefore, it is enough to estimate the mean integrated squared
error (MISE) $\mathbb{E}^{d^\dagger}\|m^\dagger-u^\dagger\|^2$ and
the trace of the posterior covariance operator $\mathcal{C}$.

{By (\ref{34}) we have}
\[m^\dagger=\mathcal{A}d^\dagger=\mathcal{A}\mathcal{L}u^\dagger+\frac{1}{\sqrt{n}}\mathcal{A}\xi.\]
Meanwhile,
\[u^\dagger=\mathcal{A}\mathcal{L}u^\dagger+(I-\mathcal{A}\mathcal{L})u^\dagger\]
so that we get the error equation
\[e:=m^\dagger-u^\dagger=\frac{1}{\sqrt{n}}\mathcal{A}\xi+(\mathcal{A}\mathcal{L}-I)u^\dagger.\]
The first part of the error comes from the noise, while the second
part comes from the regularization. {Note that for $\lambda=0$
formally we have
\[\mathcal{A}\mathcal{L}=\mathcal{C}_0\mathcal{L}^*(\mathcal{L}^*)^{-1}
\mathcal{C}_0^{-1}\mathcal{L}^{-1}\mathcal{L}=I,\] indicating that
we can make the error $e$ small by ensuring that $\lambda \ll 1$ and $n
\gg 1$. Since $\lambda=\frac{1}{n\tau^2}$ this indicates the
possibility of an optimal choice of $\tau:=\tau(n)$ to ensure that
$\lambda=\frac{1}{n\tau(n)^2}\rightarrow 0$ as $n\rightarrow\infty$
and to balance the two sources of error. In the next three theorems,
respectively, we estimate the MISE, the trace of the covariance and
the SPC.}

\begin{theorem}[MISE]
\label{t4.1}
Under Assumption \ref{a2.1} the
MISE may be estimated as follows
\begin{eqnarray}\label{45}
\rm MISE\left\{
      \begin{array}{ll}
            \asymp\frac{1}{n\lambda}(\ln\lambda^{-\frac{1}{2s}})^{-\frac{2\alpha}{b}}+
            (\ln\lambda^{-\frac{1}{2s}})^{-\frac{2\gamma}{b}},  & b\geq 1, \\
            \lesssim\frac{1}{n\lambda}(\ln\lambda^{-\frac{1}{2s}})^{-\frac{2\alpha+b-1}{b}}+
            (\ln\lambda^{-\frac{1}{2s}})^{-\frac{2\gamma}{b}},  &
            b<1.
      \end{array}
     \right.
\end{eqnarray}
\end{theorem}

\begin{proof}

{Recalling $d^\dagger=\mathcal{L}u^\dagger+\frac{1}{\sqrt{n}}\xi$ and combining with the expression above for the error $e$}, since $\xi$ is centred Gaussian, we have
\begin{eqnarray}\label{46}
\mathbb{E}^{d^\dagger}\|m^\dagger-u^\dagger\|^2
=\frac{1}{n}\mathbb{E}\|\mathcal{A}\xi\|^2+\mathbb{E}^{d^\dagger}
\|(\mathcal{A}\mathcal{L}-I)u^\dagger\|^2,
\end{eqnarray}
{from which it follows that}
\begin{eqnarray}\label{47}
&&\mathbb{E}^{d^\dagger}\|m^\dagger-u^\dagger\|^2=\frac{1}{n}{\rm Tr}(\mathcal{A}\mathcal{C}_1\mathcal{A}^*)
+\|(\mathcal{A}\mathcal{L}-I)u^\dagger\|^2\nonumber\\
&&=\frac{1}{n}\sum_{j=1}^{\infty}\frac{j^{-4\alpha-2\beta}l_j^2}{(j^{-2\alpha}l_j^2+
\lambda j^{-2\beta})^2}+\sum_{j=1}^\infty
\frac{\lambda^2j^{-4\beta}(u^{\dagger}_j)^2}{(j^{-2\alpha}l_j^2+\lambda j^{-2\beta})^2}\nonumber\\
&&=\frac{1}{n\lambda^2}\sum_{j=1}^{\infty}\frac{l_j^2j^{2\beta-4\alpha}}
{(1+\frac{1}{\lambda}l_j^2j^{2\beta-2\alpha})^2}+\sum_{j=1}^{\infty}\frac{(u^{\dagger}_j)^2}
{(1+\frac{1}{\lambda}l_j^2j^{2\beta-2\alpha})^2}\nonumber\\
&&:=\textrm{I}+\textrm{II}.
\end{eqnarray}

By Assumption \ref{a2.1}, it follows that
\begin{eqnarray*}
\textrm{I}\asymp\frac{1}{n\lambda^2}\sum_{j=1}^{\infty}\frac{\exp(-2sj^b)j^{2\beta-4\alpha}}
{(1+\frac{1}{\lambda}\exp(-2sj^b)j^{2\beta-2\alpha})^2},
\end{eqnarray*}
and
\begin{eqnarray*}
\textrm{II}\asymp\sum_{j=1}^{\infty}\frac{(u^{\dagger}_j)^2}
{(1+\frac{1}{\lambda}\exp(-2sj^b)j^{2\beta-2\alpha})^2}.
\end{eqnarray*}

To estimate $\textrm{I}$ and $\textrm{II}$ we split the sum
according to the dominating term in the denominator. Define
{\[F(x;\lambda):=\frac{1}{\lambda}\exp(-2sx^b)j^{2\beta-2\alpha}, \;x\in\mathbb{R}, \lambda>0,\] }and
note that $F(1;\lambda)>1$, for $\lambda$ sufficiently small. Since
we are considering a limit in which $\lambda \to 0$ we assume that
$F(1;\lambda)>1$ henceforth. {Let $J_\lambda$ be the unique solution
of the equation $F(x;\lambda)=1$ which exceeds $1$}. By Lemma
\ref{l4.4}, we have
\begin{equation}\label{48}
J_\lambda\sim(\ln\lambda^{-\frac{1}{2s}})^{\frac{1}{b}}.
\end{equation}

For $\textrm{I}$, if $1\leq j\leq J_\lambda$,
\begin{align}\label{485}\frac{1}{\lambda}\exp(-2sj^b)j^{2\beta-2\alpha}\leq
1+\frac{1}{\lambda}\exp(-2sj^b)j^{2\beta-2\alpha}\leq
2\frac{1}{\lambda}\exp(-2sj^b)j^{2\beta-2\alpha},\end{align} therefore

\begin{eqnarray}\label{49}
\frac{1}{n\lambda^2}\sum_{j\leq
J_\lambda}\frac{\exp(-2sj^b)j^{2\beta-4\alpha}}
{(1+\frac{1}{\lambda}\exp(-2sj^b)j^{2\beta-2\alpha})^2}
\asymp\frac{1}{n}\sum_{j\leq J_\lambda}\exp(2sj^b)j^{-2\beta}.
\end{eqnarray}
{The sum on the} right hand side
is bounded from above by the integral in the same
range, and values at both endpoints.
By Lemma \ref{l4.5}, we have
\begin{eqnarray}\label{410}
&&\frac{1}{n}\sum_{j\leq
J_\lambda}\exp(2sj^b)j^{-2\beta}\nonumber\\
&&\leq\frac{1}{n}\exp(2sJ^b_\lambda)J^{-2\beta}_\lambda+\frac{1}{n}\exp(2s)+
\frac{1}{n}\int_{1}^{J_\lambda}\exp(2sx^b)x^{-2\beta}dx\nonumber\\
&&=\frac{1}{n}\exp(2sJ^b_\lambda)J^{-2\beta}_\lambda+\frac{1}{n}\exp(2s)+
\frac{M}{n}\exp(2sJ^b_\lambda)J^{-2\beta-b+1}_\lambda(1+o(1))\nonumber\\
&&=\left\{
  \begin{array}{ll}
    \frac{M}{n}\exp(2sJ^b_\lambda)J^{-2\beta}_\lambda(1+o(1)),  & b\geq 1, \\
    \frac{M}{n}\exp(2sJ^b_\lambda)J^{-2\beta-b+1}_\lambda(1+o(1)),  & b<1,
  \end{array}
\right.
\end{eqnarray}
Since \[\frac{1}{n}\sum\limits_{j\leq
J_\lambda}\exp(2sj^b)j^{-2\beta}\geq\frac{1}{n}\exp(2sJ^b_\lambda)J^{-2\beta}_\lambda,\]
{we deduce that for, $b\geq 1$},
\begin{eqnarray}\label{new1}
\frac{1}{n}\sum_{j\leq
J_\lambda}\exp(2sj^b)j^{-2\beta}\asymp\frac{1}{n}\exp(2sJ^b_\lambda)J^{-2\beta}_\lambda
=\frac{1}{n\lambda}J^{-2\alpha}_\lambda.
\end{eqnarray}
{For $0<b<1$ we have}
\begin{eqnarray}\label{new2}
\frac{1}{n}\sum_{j\leq
J_\lambda}\exp(2sj^b)j^{-2\beta}\lesssim\frac{1}{n}\exp(2sJ^b_\lambda)J^{-2\beta-b+1}_\lambda
=\frac{1}{n\lambda}J^{-2\alpha-b+1}_\lambda.
\end{eqnarray}
If $j\geq J_\lambda$, then $1\leq
1+\frac{1}{\lambda}\exp(-2sj^b)j^{2\beta-2\alpha}\leq 2$, thus we have \begin{eqnarray*}
\frac{1}{n\lambda^2}\sum_{j>J_\lambda}\frac{\exp(-2sj^b)j^{2\beta-4\alpha}}
{(1+\frac{1}{\lambda}\exp(-2sj^b)j^{2\beta-2\alpha})^2}
\asymp\frac{1}{n\lambda^2}\sum_{j>J_\lambda}\exp(-2sj^b)j^{2\beta-4\alpha}.
\end{eqnarray*}
Under our assumption on $\lambda$ being sufficiently small, we have that $J_\lambda$ is large enough so that
 $\exp(-2sj^b)j^{2\beta-4\alpha}$ is always
decreasing with respect
to $j$ {and hence} the sum on the right hand side is bounded from above by the
integral in the same range, and the value at the left endpoint.
By Lemma \ref{l4.6}, we have
\begin{eqnarray}\label{411}
&&\frac{1}{n\lambda^2}\sum_{j>J_\lambda}\exp(-2sj^b)j^{2\beta-4\alpha}\nonumber\\
&&\leq\frac{1}{n\lambda^2}\exp(-2sJ_\lambda^b)J_\lambda^{2\beta-4\alpha}
+\frac{1}{n\lambda^2}\int_{J_\lambda}^{\infty}\exp(-2sx^b)x^{2\beta-4\alpha}dx\nonumber\\
&&\leq\frac{1}{n\lambda^2}\exp(-2sJ_\lambda^b)J_\lambda^{2\beta-4\alpha}
+\frac{M}{n\lambda^2}\exp(-2sJ_\lambda^b)J_\lambda^{2\beta-4\alpha-b+1}(1+o(1))\nonumber\\
&&=\left\{
  \begin{array}{ll}
    \frac{M}{n\lambda^2}\exp(-2sJ_\lambda^b)J_\lambda^{2\beta-4\alpha}(1+o(1)),  & b\geq 1, \\
    \frac{M}{n\lambda^2}\exp(-2sJ_\lambda^b)J_\lambda^{2\beta-4\alpha-b+1}(1+o(1)),  & b<1.
  \end{array}
\right.
\end{eqnarray}
Since
$\frac{1}{n\lambda^2}\sum\limits_{j>J_\lambda}\exp(-2sj^b)j^{2\beta-4\alpha}\geq
\frac{1}{n\lambda^2}\exp(-2sJ_\lambda^b)J_\lambda^{2\beta-4\alpha}$,
for $b\geq 1$, we have
\begin{eqnarray}
\frac{1}{n\lambda^2}\sum\limits_{j>J_\lambda}\exp(-2sj^b)j^{2\beta-4\alpha}\asymp
\frac{1}{n\lambda^2}\exp(-2sJ_\lambda^b)J_\lambda^{2\beta-4\alpha}
=\frac{1}{n\lambda}J_\lambda^{-2\alpha},
\end{eqnarray}
and for $0<b<1$,
\begin{align}
\frac{1}{n\lambda^2}\sum\limits_{j>J_\lambda}\exp(-2sj^b)j^{2\beta-4\alpha}
\lesssim\frac{1}{n\lambda^2}\exp(-2sJ_\lambda^b)J_\lambda^{2\beta-4\alpha-b+1}
=\frac{1}{n\lambda}J_\lambda^{-2\alpha-b+1}.
\end{align}

{To estimate $\textrm{II}$, we employ an analysis similar to
that applied to $\textrm{I}$}. By (\ref{485}) we have
\begin{eqnarray}\label{412}
&&\sum_{j\leq J_\lambda}\frac{(u^{\dagger}_j)^2}
{(1+\frac{1}{\lambda}\exp(-2sj^b)j^{2\beta-2\alpha})^2}\asymp\sum_{j\leq
J_\lambda}(u^{\dagger}_j)^2
\lambda^2\exp(4sj^b)j^{4\alpha-4\beta}\nonumber\\
&&=\sum_{j\leq J_\lambda}j^{2\gamma}(u^{\dagger}_j)^2
\lambda^2\exp(4sj^b)j^{4\alpha-4\beta-2\gamma}.
\end{eqnarray}
For $\lambda$ small enough, the terms
$\exp(4sj^b)j^{4\alpha-4\beta-2\gamma}$ for $1\leq j\leq J_\lambda$
are dominated by
$\exp(4sJ_\lambda^b)J_\lambda^{4\alpha-4\beta-2\gamma}$, so we have
the following upper bound for the sum (\ref{412}):
\begin{eqnarray*}
\sum_{j\leq J_\lambda}j^{2\gamma}(u^{\dagger}_j)^2
\lambda^2\exp(4sj^b)j^{4\alpha-4\beta-2\gamma}
\leq\lambda^2\exp(4sJ_\lambda^b)J_\lambda^{4\alpha-4\beta-2\gamma}\|u^\dagger\|_\gamma^2.
\end{eqnarray*}
{Furthermore}
\begin{equation*}
\sum\limits_{j\leq J_\lambda}j^{2\gamma}(u^{\dagger}_j)^2
\lambda^2\exp(4sj^b)j^{4\alpha-4\beta-2\gamma}\geq
(u^\dagger_{J_\lambda})^2\lambda^2\exp(4sJ_\lambda^b)J_\lambda^{4\alpha-4\beta-2\gamma},
\end{equation*}
{implying that, since $\gamma>0$ and $u \in H^{\gamma}$,}
\begin{equation}\label{413}
\sum_{j\leq J_\lambda}j^{2\gamma}(u^{\dagger}_j)^2
\lambda^2\exp(4sj^b)j^{4\alpha-4\beta-2\gamma}\asymp
\lambda^2\exp(4sJ_\lambda^b)J_\lambda^{4\alpha-4\beta-2\gamma}
=J_\lambda^{-2\gamma}.
\end{equation}
The other part of the sum $\textrm{II}$ satisfies
\begin{eqnarray*}
\sum_{j>J_\lambda}\frac{(u^{\dagger}_j)^2}
{(1+\frac{1}{\lambda}\exp(-2sj^b)j^{2\beta-2\alpha})^2}\asymp\sum_{j>
J_\lambda}(u^{\dagger}_j)^2=\sum_{j>J_\lambda}j^{2\gamma}(u^{\dagger}_j)^2j^{-2\gamma}.
\end{eqnarray*}
It {follows} that
\begin{eqnarray}\label{414}
\sum_{j>J_\lambda}j^{2\gamma}(u^{\dagger}_j)^2j^{-2\gamma}\asymp
J_\lambda^{-2\gamma},
\end{eqnarray}
since $u \in H^{\gamma}.$

Combining (\ref{46}) - (\ref{414}) completes the proof.
\end{proof}

\begin{theorem}[Trace of $\mathcal{C}$]
\label{t4.2}

{Let Assumption \ref{a2.1} hold and consider
the posterior covariance operator
$\mathcal{C}$ given by (\ref{26}), with $\lambda$ as in}
(\ref{28}). Then the trace is estimated as
\begin{equation}\label{415}
{\rm Tr}(\mathcal{C})\asymp\frac{1}{n\lambda}(\ln\lambda^{-\frac{1}{2s}})^{-\frac{2\alpha-1}{b}}.
\end{equation}
\end{theorem}

\begin{proof}

From (\ref{33}) and (\ref{35}) we have
\begin{eqnarray}\label{416}
{\rm Tr}(\mathcal{C})=\sum_{j=1}^{\infty}\frac{\tau^2\lambda
c_{0j}c_{1j}}{c_{0j}l_j^2+\lambda
c_{1j}}\asymp\frac{1}{n\lambda}\sum_{j=1}^{\infty}\frac{j^{-2\alpha}}{1+\frac{1}{\lambda}
\exp(-2sj^b)j^{2\beta-2\alpha}}.
\end{eqnarray}
As in the proof of Theorem \ref{t4.1} we split the sum according to the
dominating term in the denominator. For the first part, using equation (\ref{485}), we have
\begin{eqnarray}\label{417}
\frac{1}{n\lambda}\sum_{j\leq
J_\lambda}\frac{j^{-2\alpha}}{1+\frac{1}{\lambda}\exp(-2sj^b)j^{2\beta-2\alpha}}
\asymp\frac{1}{n}\sum_{j\leq J_\lambda}\exp(2sj^b)j^{-2\beta},
\end{eqnarray}
{where the behaviour of the right hand side is given by equations (\ref{new1}) and (\ref{new2}).}
The other part of the sum on the right hand side of (\ref{416}) satisfies
\begin{eqnarray*}
\frac{1}{n\lambda}\sum_{j>J_\lambda}\frac{j^{-2\alpha}}{1+\frac{1}{\lambda}\exp(-2sj^b)j^{2\beta-2\alpha}}
\asymp\frac{1}{n\lambda}\sum_{j>J_\lambda}j^{-2\alpha}.
\end{eqnarray*}
By \cite[Lemma 6.2]{KVZ12}, the last sum can be {estimated} as
\begin{eqnarray*}
\sum_{j>J_\lambda}j^{-2\alpha}\asymp J_\lambda^{-2\alpha+1},
\end{eqnarray*}
hence
\begin{eqnarray}\label{418}
\frac{1}{n\lambda}\sum_{j>J_\lambda}\frac{j^{-2\alpha}}{1+\frac{1}{\lambda}\exp(-2sj^b)j^{2\beta-2\alpha}}
\asymp\frac{1}{n\lambda}J_\lambda^{-2\alpha+1}.
\end{eqnarray}

Combining (\ref{48}), (\ref{416})-(\ref{418}) completes the
proof.
\end{proof}

{We combine the two preceding theorems to determine
the overall contraction rate.}

\begin{theorem}[Rate of Contraction]
\label{t4.3}
Suppose that Assumption \ref{a2.1} holds, $\lambda$ is given by (\ref{28}) and $\tau(n)>0$
satisfies $n\tau^2(n)\rightarrow \infty$. Then the posterior
distribution $\mu_{\lambda,n}^{d^\dagger}$ contracts around the true
solution $u^\dagger$ at the rate
\begin{equation}\label{419}
\epsilon_n=\big(\ln(n\tau^2{(n)})\big)^{-\frac{\gamma}{b}}+
\tau(n)\big(\ln(n\tau^2{(n)})\big)^{-\frac{\alpha-\frac{1}{2}}{b}}.
\end{equation}
In particular, since the rate is undetermined up to a multiplicative constant independent of $n$, we may take
\begin{eqnarray}\label{420}
\epsilon_n=\left\{
      \begin{array}{ll}
            \Big(\ln n\Big)^{-\frac{\gamma\wedge(\alpha-\frac{1}{2})}{b}},
            & \tau(n)\equiv 1,\\
            \Big(\ln n\Big)^{-\frac{\gamma}{b}},
            & n^{-\frac{1}{2}+\sigma}\lesssim\tau(n)\lesssim (\ln
            n)^{\frac{\alpha-\gamma-\frac{1}{2}}{b}},
      \end{array}
     \right.
\end{eqnarray}
where $\sigma>0$ is some constant.
\end{theorem}

\begin{proof}

The estimate (\ref{419}) follows by combining (\ref{44}), Theorem
\ref{t4.1} and Theorem \ref{t4.2}. The rate for $\tau(n)\equiv1$ follows immediately. 
{In the case of varying $\tau(n)$, observe that in order to balance the contributions of the two terms in (\ref{419}), $\tau(n)$ needs to be large enough so that $n\tau^2(n)\rightarrow\infty$ as $n\rightarrow\infty$, but small enough so that the second term is bounded by the first one. Since the function $\big(\ln(\cdot)\big)^{-\kappa}$, $\kappa>0$ is 
decreasing, this can be achieved by choosing $n^{-\frac{1}{2}+\sigma}\lesssim\tau(n)\lesssim (\ln n)^{\frac{\alpha-\gamma-\frac{1}{2}}{b}}$ for some constant $\sigma>0$, 
in which case the rate becomes
\begin{eqnarray*}
\epsilon_n&&\lesssim\big(\ln(n\cdot n^{-1+2\sigma})\big)^{-\frac{\gamma}{b}}+ 
(\ln n)^{\frac{\alpha-\gamma-\frac{1}{2}}{b}}\big(\ln(n\cdot n^{-1+2\sigma})\big)^{-\frac{\alpha-\frac{1}{2}}{b}}\nonumber\\
&&=\big(2\sigma\big)^{-\frac{\gamma}{b}}\big(\ln n\big)^{-\frac{\gamma}{b}}+(2\sigma)^{-\frac{\alpha-\frac{1}{2}}{b}}\big(\ln n\big)^{-\frac{\gamma}{b}}\nonumber\\
&&\lesssim \Big(\ln n\Big)^{-\frac{\gamma}{b}}.
\end{eqnarray*}
This completes the proof.}
\end{proof}

\begin{remark}\label{ch2:rem}

\item[i)]The rate of the MISE is determined by the regularity of the prior $\alpha$,
the regularity of the truth $\gamma$ and the degree of ill-posedness
as determined by the power $b$ in the eigenvalues of $\Lc$ ($s$ does
not affect the rate). On the other hand, the rate of the trace of
the posterior covariance is determined by $\alpha$ and $b$ and has
nothing to do with the regularity of the truth $\gamma$. Finally the
rate of contraction is determined by $\alpha, \gamma$ and $b$.
Observe that the regularity of the noise, $\beta$, does not affect
the rate. In the case of mildly ill-posed problems where the
singular values of ${\mathcal L}$ decay algebraically $\beta$ does
appear in the error estimates, but only through the difference in
regularity between the forward operator and the noise covariance
\cite{ALS12}. For our severely ill-posed problem this difference may
be thought of as being infinite, explaining why $\beta$ disappears
from the error estimates here.

\item[ii)] For fixed $\tau=1$, the
rate of contraction is $ \Big(\ln
n\Big)^{-\frac{\gamma\wedge(\alpha-\frac{1}{2})}{b}}$, that is, as
$\gamma$ grows the rate improves until $\gamma=\alpha-\frac12$, at
which point the rate saturates at $ \Big(\ln
n\Big)^{-\frac{\alpha-\frac{1}{2}}{b}}$. {Note that the saturation point $\gamma=\alpha-\frac12$ is also the crossover point between the true solution being in the support of the prior (prior oversmoothing) or not (prior undersmoothing).} On the contrary, for
$n^{-\frac{1}{2}+\sigma}\lesssim\tau\lesssim (\ln
            n)^{\frac{\alpha-\gamma-\frac{1}{2}}{b}}$ the rate is
            $(\ln n)^{-\frac{\gamma}b}$ and never saturates.
\item[iii)] For the appropriate choice of $\tau=\tau(n)$ the contraction
rate is $\epsilon_n=(\ln n)^{-\frac{\gamma}b},$ which is optimal in the minimax sense with $L^2$-loss \cite{Cav08,
KVZ12}. The minimax rate is also achieved if we have fixed $\tau\equiv1,$ provided the prior is oversmoothing, $\gamma\leq\alpha-\frac12$.
\end{remark}

{
We conclude the section with several technical lemmas
used in the proof of the preceding theorems.}

\begin{lemma}
\label{l4.4}

Let $a,b>0$ and $t\in\mathbb{R}$ be constants. For all
$\lambda$ sufficiently small the equation
\begin{equation}\label{421}
\frac{1}{\lambda}\exp(-ax^b)x^{t}=1,
\end{equation}
has a unique solution $J_{\lambda}$ in  $\{x \ge 1\}$
and
$J_\lambda\sim(\ln\lambda^{-\frac{1}{a}})^{\frac{1}{b}}$ as
$\lambda\rightarrow 0$.

\end{lemma}

\begin{proof}
Uniqueness of a root in $\{x \ge 1\}$ follows automatically provided
\[\lambda^{-1}\exp(-a)>1,\] since $x \mapsto \exp(-ax^b)x^{t}$ has at
most one maximum in $\{x \ge 0\}.$ From (\ref{421}), it is easy to
{see} that
\begin{equation*}
1=\frac{\ln\lambda^{-\frac{1}{a}}}{J_\lambda^b}+\frac{t}{a}\frac{\ln
J_\lambda}{J_\lambda^b}.
\end{equation*}
{Since we are looking for solutions in $\{x\geq 1\}$, we have that $\ln{J_\lambda}\geq0$ hence $J_\lambda\rightarrow\infty$ as $\lambda\rightarrow
0$.} This implies
$1\sim\frac{\ln\lambda^{-\frac{1}{a}}}{J_\lambda^b}$, which
completes the proof.
\end{proof}

\begin{lemma}
\label{l4.5}
{
For $a>0, b>0$ and $ c\in\mathbb{R}$, we have as
$J\rightarrow\infty$,
\begin{equation}\label{422}
\int_1^J\exp(ax^b)x^{c}dx\sim\frac{1}{ab}\exp(aJ^b)J^{c-b+1}.
\end{equation}}
\end{lemma}

\begin{proof}

{By variable substitution $x^b=y$ and integration by parts, we get
\begin{eqnarray}\label{423}
&&\int_1^J\exp(ax^b)x^{c}dx\nonumber\\
&&=\frac{1}{ab}(\exp(aJ^b)J^{c-b+1}-\exp(a))-
\frac{c-b+1}{ab^2}\int_1^{J^b}\exp(ay)y^{\frac{c-2b+1}{b}}dy,\nonumber\\
\end{eqnarray}
thus letting $I(J):=\int_1^{J^b}\exp(ay)y^{\frac{c-2b+1}{b}}dy$, we have that it suffices to show that \begin{align}\label{ch2:clim}\lim_{J\to\infty}\frac{I(J)}{\exp(aJ^b)J^{c-b+1}}=0.\end{align}
Indeed, if $c-2b+1\geq0$ then we have
\begin{equation*}
\frac{\exp(ay)y^{\frac{c-2b+1}{b}}}{\exp(aJ^b)J^{c-b+1}}\leq\exp(a(y-J^b))J^{-b},
\end{equation*}
and (\ref{ch2:clim}) holds.
If $c-2b+1<0$, we use the variable substitution $e^{ay}=z$ to get that
\begin{align*}I(J)=\frac{1}{a^\frac{c-b+1}{b}}\int_{e^a}^{e^{aJ^b}}(\ln(z))^{\frac{c-2b+1}{b}}dz.\end{align*} By Lemma \ref{ch2:clem} below, we then have that \[I(J)\lesssim{\exp(aJ^b)}{J^{c-2b+1}},\]
hence (\ref{ch2:clim}) holds.}
\end{proof}

\begin{lemma}
\label{l4.6}
For $J>0, a>0, b>0$ and $c\in\mathbb{R}$ we have
\begin{equation}\label{424}
\int_J^\infty\exp(-ax^b)x^{c}dx\lesssim\exp(-aJ^b)J^{c-b+1}.
\end{equation}
\end{lemma}

\begin{proof}
By variable substitution $x^b=y$ and integration by parts, we have
\begin{eqnarray*}
&&\int_J^\infty\exp(-ax^b)x^{c}dx\nonumber\\
&&=\frac{1}{ab}\exp(-aJ^b)J^{c-b+1}+\frac{c-b+1}{ab^2}\int_{J^b}^\infty\exp(-ay)y^{\frac{c-2b+1}{b}}dy.
\end{eqnarray*}
If $\frac{c-b+1}{ab^2}>0$, then we integrate by parts for $n$ times
until $\frac{c-nb+1}{ab^2}<0$ for the first time. When the constant
in front of the integral finally becomes negative we can ignore the
integral on the right hand side to get
\begin{eqnarray*}
\int_J^\infty\exp(-ax^b)x^{c}dx\leq\frac{1}{ab}\exp(-aJ^b)(J^{c-b+1}(1+o(1))).
\end{eqnarray*}
\end{proof}

\begin{lemma}\label{ch2:clem}
{For any $q, a>0$ we have as $x\to\infty$ \begin{align*}\int_{e^a}^x\frac{dz}{(\ln(z))^q}\leq\frac{x}{(\ln(x))^{q}}(2+o(1)).\end{align*}}
\end{lemma}
\begin{proof}
{We split the integral as follows
\begin{align}\label{eq:claim}\int_{e^a}^x\frac{dz}{(\ln(z))^q}&=\int_{e^a}^{e^{2q}}\frac{dz}{(\ln(z))^q}+\int_{e^{2q}}^x\frac{dz}{(\ln(z))^q}\nonumber\\
&=c(q,a)+\int_{e^{2q}}^x\frac{dz}{(\ln(z))^q},
\end{align}
where $c(q,a)$ is a real constant. For $z\geq e^{2q}$ it holds \begin{align*}\ln(z)\geq 2q,\end{align*} hence
dividing by $(\ln(z))^{q+1}$ and rearranging terms we get that \begin{align}\label{eq:ineq}\frac{q}{(\ln(z))^{q+1}}\leq\frac{1}{2(\ln(z))^q}.\end{align}
Integration by parts in the integral on the right hand side of (\ref{eq:claim}) gives
\begin{align*}\int_{e^{2q}}^x\frac{dz}{(\ln(z))^q}=\frac{x}{(\ln(x))^q}-\frac{e^{2q}}{(2q)^q}+\int_{e^{2q}}^x\frac{q}{(\ln(z))^{q+1}}dz,\end{align*}
hence using (\ref{eq:ineq}) and rearranging terms, we have
\begin{align*}
\int_{e^{2q}}^x\frac{dz}{(\ln(z))^q}&\leq 2\frac{x}{(\ln(x))^q}-2\frac{e^{2q}}{(2q)^q}\\
&=2\frac{x}{(\ln(x))^q}+\tilde{c}(q).
\end{align*}
Concatenating we obtain the result.}
\end{proof}

\section{Example}
\label{s:5}

In this section, we present the Cauchy problem for the
Helmholtz equation as an example to which the
{theoretical} analysis of
this paper can be applied. For simplicity, we only consider the
small wave number case ($0<k<1$) for illustration. For more details
regarding the more general case, we refer to \cite{ZFD12}.

Consider the following {boundary value} problem
for the Helmholtz equation:
\begin{equation}\label{51}
\left\{
  \begin{array}{ll}
    \Delta v(x,y)+k^2 v(x,y)=0,
    & (x,y)\in (0,\pi)\times(0,1), \\
    v_y(x,0)=0,  & x\in [0,\pi], \\
    v(x,1)=u(x),  & x\in [0,\pi], \\
    v(0,y)=v(\pi,y)=0, & y\in [0,1].
  \end{array}
\right.
\end{equation}
Problem (\ref{51}) is well-posed {since it corresponds to
inversion of a negative-definite ellipic operator with
mixed Dirichlet/Neumann data.}
In fact, by the method of separation of variables,
the solution $v(x,y)$ in domain $(0,\pi)\times (0,1)$ can be
expressed as
\begin{equation}\label{52}
v(x,y)=\sum\limits_{j=1}^{\infty}\frac{\cosh(y\sqrt{j^2-k^2})}{\cosh(\sqrt{j^2-k^2})}u_j\varphi_j(x),
\end{equation}
where $\varphi_j(x)=\sqrt{\frac{2}{\pi}}\sin(jx)$ and $u_j=\langle
u, \varphi_j\rangle$.

Define the forward mapping
$\mathcal{L}:\mathscr{D}(\mathcal{L})\subset L^2(0,\pi)\rightarrow
L^2(0,\pi)$ {by} \[\mathcal{L}
u(x)=v(x,0)=\sum\limits_{j=1}^{\infty}\frac{1}{\cosh(\sqrt{j^2-k^2})}u_j
\varphi_j(x),\] {which maps the boundary data of (\ref{51}) on $y=1$
into the solution on $y=0.$ Then $\mathcal{L}$ is a self-adjoint,
positive-definite, linear operator, with eigenvalues behaving as}
\begin{equation}\label{53}
l_j=\frac{1}{\cosh(\sqrt{j^2-k^2})}\sim\exp(-j).
\end{equation}

{The inverse problem is to find the function $u$,
given noisy observations of $v(\cdot,0).$
More precisely the data $d$ is given by
\begin{align*}
d&=v(\cdot,0)+\frac{1}{\sqrt{n}}\xi,\\
&={\mathcal L}u+\frac{1}{\sqrt{n}}\xi.
\end{align*}
If we place a Gaussian measure $N(0, \tau^2\C_0)$ as prior on $u$ and
assume that $\xi$ is also Gaussian $N(0,\C_1)$,
then we may apply the theory developed in this paper.}
Under Assumption \ref{a2.1}, Theorem
\ref{t4.3} can be applied to this problem with $b=1$ and $s=1$ to obtain the
contraction rate of the conditional Gaussian posterior distribution.

We now present a numerical simulation for obtaining the rate of the
MISE as the noise disappears ($n\to\infty$), when $\alpha=2,
\gamma=1$ and we have a fixed $\tau=1$. In this case, our theory
predicts that
\[\rm MISE\asymp\big(\ln(\sqrt{n})\big)^{-2(\alpha\wedge\gamma)}=\big(\ln(\sqrt{n})\big)^{-2}.\]
To simulate MISE we average the error over a thousand realizations
of the noise $\xi$, for $n=10^k, \;k=1,...,100$. We denote the
simulated MISE by $\widehat{\rm MISE}$. The true solution
$u^\dagger\in H^\gamma$ is a fixed draw from a Gaussian measure
$N(0,\Sigma)$, where $\Sigma$ has eigenvalues
$\sigma_j=j^{-2\gamma-1-\varepsilon},$ for $\varepsilon=10^{-10}$.
We use the first  $10^5$ Fourier modes. In Figure \ref{fig} we plot
$-\frac12\ln\big(\widehat{\rm MISE}\big)$ against
$\ln\big(\ln(\sqrt{n})\big)$ in the case $ \beta=0$. The solid line
is the relation predicted by Theorem \ref{t4.1}, that is, a line
with slope $1$. A least square fit to the simulated points gives a
slope of $1.0341$ with coefficient of determination $0.9884$.
\begin{figure}
\begin{center}
  \includegraphics[width=9cm]{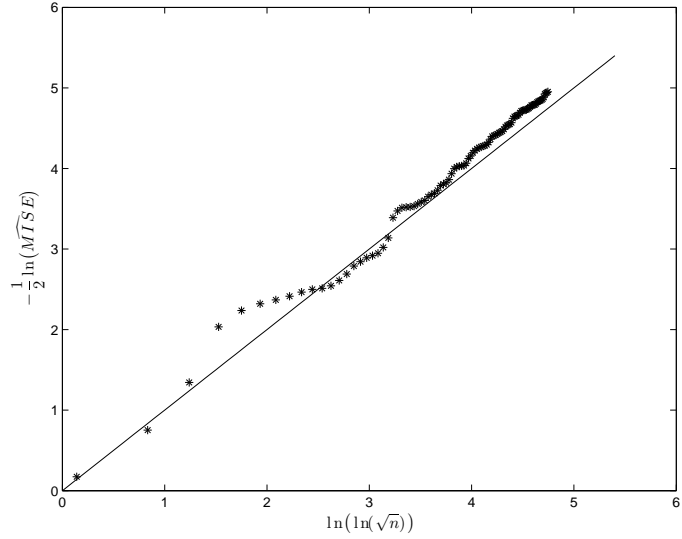}
  \vspace*{-8pt}
  \caption{\small{$-\frac12\ln\big(\widehat{\rm MISE}\big)$ plotted against
            $\ln\big(\ln(\sqrt{n})\big)$ for $n=10^k, \;k=1,...,100$ in the case
            $b=s=1, \alpha=2, \beta=0, \gamma=1$, for fixed $\tau=1$.}}\label{fig}
\end{center}
\end{figure}
In Figure \ref{fig2} we have $\beta=2$ and all the other parameters
the same. The least squares fit gives a slope $0.9723$ with
coefficient of determination $0.9916$, confirming that the
regularity of the noise as determined by $\beta$ does not affect the
rate of convergence.
\begin{figure}
\begin{center}
  \includegraphics[width=9cm]{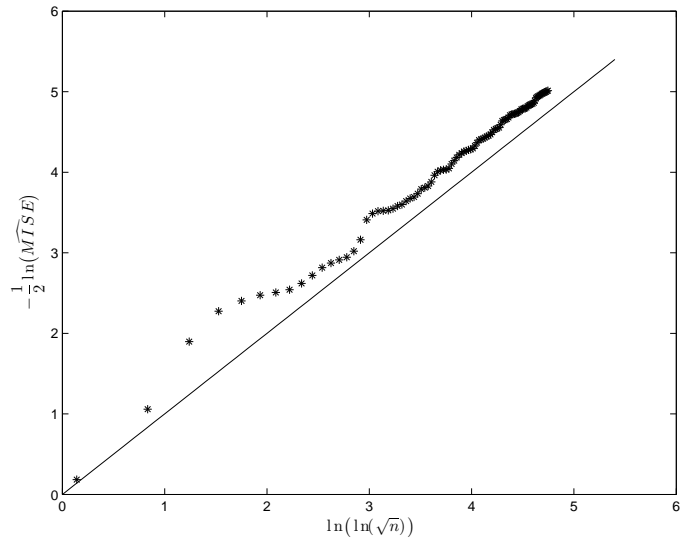}
  \vspace*{-8pt}
  \caption{\small{$-\frac12\ln\big(\widehat{\rm MISE}\big)$ plotted against
            $\ln\big(\ln(\sqrt{n})\big)$ for $n=10^k, \;k=1,...,100$ in the case
            $b=s=1, \alpha=2, \beta=2, \gamma=1$, for fixed $\tau=1$.}}\label{fig2}
\end{center}
\end{figure}

\vspace{4mm}

{\bf Acknowledgmenets.} {The authors are grateful to
Bartek Knapik and Harry Van Zanten for helpful discussions. {The authors are also grateful to two referees for several very useful comments}.}

\bibliographystyle{plain}
\bibliography{Expon-IP}
\end{document}